\documentclass{amsart}

\usepackage{amssymb, latexsym,pdfsync,amsmath,amsthm, ulem,hyperref,graphicx}
\usepackage{fullpage}
\usepackage{pgf,tikz}
\usepackage{mathrsfs}
\usetikzlibrary{arrows}

\usepackage{longtable}

\newtheorem{theorem}{Theorem}

\theoremstyle{definition}
\newtheorem{definition}{Definition}

\newcommand{\bS}{\mathbb{S}}

\newcommand{\bF}{\mathbb{F}}

\newcommand{\Fq}{\mathbb{F}_q}

\newcommand{\cV}{\mathcal V}

\def\PG{\mathrm{PG}}
\def\PGL{\mathrm{PGL}}
\def\GL{\mathrm{GL}}

\def\semsymmat#1#2#3#4#5#6#7#8#9{%
  \def\ArgI{{#1}}%
  \def\ArgII{{#2}}%
  \def\ArgIII{{#3}}
  \def\ArgIV{{#4}}
  \def\ArgV{{#5}}
  \def\ArgVI{{#6}}
  \def\ArgVII{{#7}}
  \def\ArgVIII{{#8}}
  \def\ArgIX{{#9}}%
  \BlahRelay
}
\def\BlahRelay#1{%
  \left [ \begin{matrix} \ArgI&\ArgII&\ArgIII&\ArgIV \\  \ArgII&\ArgV&\ArgVI&\ArgVII \\  \ArgIII&\ArgVI&\ArgVIII&\ArgIX \\  \ArgIV&\ArgVII&\ArgIX&#1 \end{matrix} \right ]%
}

\def\BlahRelay#1{%
  \left [ \begin{matrix} \ArgI&\ArgII&\ArgIII&\ArgIV \\  &\ArgV&\ArgVI&\ArgVII \\  &&\ArgVIII&\ArgIX \\  &&&#1 \end{matrix} \right ]%
}

\begin{document}
\title{Symplectic 4-dimensional semifields of order $8^4$ and $9^4$}
\author{Michel Lavrauw and John Sheekey}\thanks{The first author acknowledges the support of {\it The Scientific and Technological Research Council of Turkey}, T\"UB\.{I}TAK (project no. 118F159).}

\maketitle

\begin{abstract}
We classify symplectic 4-dimensional semifields over $\bF_q$, for $q\leq 9$, thereby extending (and confirming) the previously obtained classifications for $q\leq 7$. The classification is obtained by classifying all symplectic semifield subspaces in $\PG(9,q)$ for $q\leq 9$ up to $K$-equivalence, where $K\leq \PGL(10,q)$ is the lift of $\PGL(4,q)$ under the Veronese embedding of $\PG(3,q)$ in $\PG(9,q)$ of degree two.  Our results imply the non-existence of non-associative symplectic 4-dimensional semifields for $q$ even, $q\leq 8$. For $q$ odd, and $q\leq 9$, our results imply that the isotopism class of a symplectic non-associative 4-dimensional semifield over $\bF_q$ is contained in the Knuth orbit of a Dickson commutative semifield.
\end{abstract}
\section{Introduction}

A finite semifield is a division algebra in which multiplication is not assumed to be associative. Finite semifields were introduced by Leonard Eugene Dickson at beginning of the 20th century and are well studied objects in algebra, finite geometry and combinatorics, with applications in coding theory and cryptography, through their connections with projective planes, MRD codes, PN functions and S-boxes. Commutative semifields are of special interest because the Knuth orbits of their isotopism classes contain the isotopism classes of symplectic semifields which in turn give rise to symplectic spreads, Dembowski-Ostrom polynomials, and Hadamard difference sets. There are many constructions of semifields known, most of which are non-commutative. We refer to \cite[Section 6]{LaPo2011} for an overview of these constructions and the above mentioned links.

Computational classifications of semifields ($n$-dimensional over the finite field of order $q$) have been obtained for $q=2, n\leq 6$; $q=3, n\leq 5$; and $q\leq 5, n\leq 4$. These classifications have a (relatively) long history; the first of these results was obtained in the early days of computational results in mathematics. E. Kleinfeld \cite{Kleinfeld1960} classified semifields (at that time called Veblen-Wedderburn systems) of order 16 with the help of a computer in 1960. D. E. Knuth used a computer to construct a new semifield of order 32, see \cite{Knuth1965}. A computational classification of all semifields of order 32 was obtained in 1962 (using a CDC 1604) by Walker \cite{Walker1962}, and confirmed by Knuth \cite{Knuth1965} soon after. More recently, computational classifications of semifields were obtained by Dempwolff \cite{Demp2008} ($q^n=3^4$), and by Combarro, Ranilla, and R\'ua, see \cite{RuCoRa2009} ($q^n=2^6$), \cite{RuCoRa2011} ($q^n=4^4, 5^4$), \cite{RuCoRa2012} ($q^n=3^5$), and \cite{CoRuRa2012} ($q^n=7^4$). 

Complete computational classifications are also known for certain special type of semifields, see for example \cite{LaRo2018}, in which all 8-dimensional rank 2 commutative semifields are classified, relying on theoretical results from \cite{BlBaLa2003} and \cite{Lavrauw2006}.
A partial classification of commutative semifields of order $3^5$ and $5^5$ can be found in \cite{CoKo2010}, where the authors restrict the classification to commutative semifields whose corresponding Dembowski-Ostrom polynomials have their coefficients in the base field. 

The contribution of this paper is the classification of symplectic semifields of order $8^4$ and $9^4$ with centre containing the fields of order $8$ and $9$ respectively. This advances the computational classification of symplectic semifields and of commutative semifields with the following theorems as our main results.

{\bf Theorem 1.}
{\it There exist no non-associative 4-dimensional symplectic semifields over $\bF_q$, $q$ even, $q\leq 8$.}

{\bf Theorem 2.}
{\it A non-associative 4-dimensional symplectic semifield over $\bF_q$, $q$ odd, $q\leq 9$ is Knuth-equivalent to a Dickson commutative semifield.}

As a corollary we have the following classifications in terms of commutative semifields.

{\bf Corollary 1.}
{\it There exist no non-associative 4-dimensional commutative semifields over $\bF_q$, $q$ even, $q\leq 8$.}

{\bf Corollary 2.}
{\it A non-associative 4-dimensional commutative semifield over $\bF_q$, $q$ odd, $q\leq 9$ is isotopic to a Dickson commutative semifield.}

The paper is organised as follows. In Section  \ref{sec:defs} we introduce the necessary notation and terminology. Section \ref{sec:sss} contains the description of the algorithms and the proof of the main results. In Section \ref{sec:sslines} and Section \ref{sec:ssplanes} we list the representatives of the $K$-orbits of the symplectic semifield lines and of the $K$-orbits of two-dimensional symplectic semifield subspaces, under the action of $K\cong \PGL(4,q)$ in its action on the subspaces of $\PG(9,q)$, for $q=8$ and $q=9$.

\section{Symplectic semifields and the associated algebraic varieties}
\label{sec:defs}

In this section we set up some of the notation and terminology used in the paper. For reasons of brevity, we do not include all of the basic definitions in the theory of semifields; for those we refer the reader to \cite{LaPo2011}.

Throughout this paper, let $q$ be a power of a prime, and denote the finite field of order $q$ by $\bF_q$. We consider semifields of dimension $n$ over $\Fq$; that is, division algebras in which multiplication is not assumed to be associative, whose centre contains $\Fq$, and which are $n$-dimensional  over $\Fq$. Such semifields are in one-to-one correspondence with {\it semifield spread sets}: $n$-dimensional subspaces of $(n\times n)$-matrices over $\Fq$ in which every nonzero element is invertible. We note that we are not assuming a multiplicative identity in the definition of semifield; usually a division algebra without identity is called a {\it presemifield}, but we omit this distinction here for convenience and brevity. That we may do so without loss of generality follows from a well-known result which says that every presemifield is isotopic to a semifield, and the fact that we are interested in the classification of semifields up to isotopism.

It is also well-known that each isotopism class $[\bS]$ of a semifield $\bS$ gives rise to a total of six isotopism classes of semifields, called the {\it Knuth orbit of $\bS$}. Semifields whose isotopism class belongs to the same Knuth-orbit are called {\it Knuth-equivalent}.
A semifield is said to be {\it commutative} if multiplication is commutative. A semifield is said to be {\it symplectic} if the spread it defines in $\PG(2n-1,q)$ (the $(2n-1)$-dimensional projective space over $\bF_q$) consists of totally singular subspaces with respect to some symplectic polarity; i.e. it defines a spread of the a symplectic polar space $W(2n-1,q)$. It is well known that a semifield is Knuth-equivalent to a commutative semifield if and only if it is Knuth-equivalent to a symplectic semifield. 

In order to make the computational classification feasible we used a geometric approach. Isotopism classes of semifields can be seen as orbits (under a well-defined group action) of subspaces in the space of a Segre variety in a projective space \cite{Lavrauw2008}. Applying this connection to 4-dimensional semifields, we end up in a 15-dimensional projective space, which is unfortunately still too large for the parameters of the semifields that we are aiming to classify.  However, since we are interested in commutative (or symplectic) semifields, we can reduce our search space to the space of a Veronese variety \cite{Lavrauw2008,LuMaPoTr2011}, which in this case is a 9-dimensional projective space over $\bF_q$. Below, we give some details of this connection between the algebraic approach and the geometric approach to the isotopism problem for symplectic semifields.

By an appropriate choice a of non-degenerate symplectic bilinear form in $2n$ variables, it can be shown that symplectic semifields are in one-to-one correspondence with $n$-dimensional subspaces of {\it symmetric} $(n\times n)$-matrices over $\Fq$ in which every nonzero element is invertible \cite{Kantor2003}. Since the condition for a matrix to be invertible is invariant under multiplication by a non-zero scalar, it is natural to work in the projective setting, so we denote $N={n+1\choose 2}$ and consider $(n-1)$-dimensional subspaces of $\PG(N-1,q)$. In this projective setting, the points defined by rank one symmetric matrices over $\bF_q$ correspond to the $\bF_q$-rational points of a {\it Veronese variety}, the image of the {\it Veronese map} from $\PG(n-1,q)$ into $\PG(N-1,q)$. We will denote points of a projective space by a letter followed by a tuple of field elements; for example, $p(1,0,0,0)$ denotes a point of $\PG(3,q)$, whose underlying coordinate vector is $(1,0,0,0)$. We identify a point defined by an $(n\times n)$ symmetric matrix $A=(a_{ij})$ with the point $p$ of $\PG(N-1,q)$ given by $p(a_{11},a_{12},\ldots,a_{1n},a_{22},\ldots,a_{2n},\ldots,a_{nn})$; that is, as coordinates we take the upper triangular part of the matrix and then read the entries row by row. We also define the {\it rank of the point $p$} as the the rank of the corresponding matrix $A$.

Given this setup, the Veronese variety $\cV_{n-1,2}(q)$ is the image of the Veronese map $\nu_{n-1,2}$ defined by
\begin{align*}
\nu_{n-1,2}&: \PG(n-1,q)\mapsto \PG(N-1,q)\\
&: s(a_1,a_2,\ldots,a_n)\mapsto p(a_1^2,a_1a_2,\ldots,a_1a_n,a_2^2,\ldots,a_2a_n,\ldots,a_n^2),
\end{align*}
where $N={{n+2}\choose{2}}-1$.
That is, the point defined by the row vector $a$ is mapped to the point defined by the entries of the matrix $a^T a$. It is easy to see that $\bF_q$-rational points on $\cV_{n-1,2}(q)$ correspond precisely to points defined by symmetric matrices of rank one under the above mentioned correspondence.

Points of rank two in $\PG(N,q)$ are points which belong to the secant variety $\cV_{n,2}(q)^{(2)}$ of the Veronese variety $\cV_{n,2}(q)$ but do not belong to $\cV_{n,2}(q)$. Similarly, one has the $i$-th secant variety $\cV_{n,2}(q)^{(i)}$ of $\cV_{n,2}(q)$. It follows that symplectic semifield spread sets correspond to $(n-1)$-dimensional subspaces of $\PG(N-1,q)$ disjoint from the $(n-2)$-nd secant variety $\cV_{n-1,2}(q)^{(n-1)}$ of the Veronese variety $\cV_{n-1,2}(q)$.  Since we will search for symplectic semifield spread sets by classifying all subspaces (of any dimension) disjoint from $\cV_{n-1,2}(q)^{(n-1)}$, we introduce the following terminology for convenience.

\begin{definition}
A {\it symplectic semifield subspace} of $\PG(N-1,q)$ is a subspace which is disjoint from the $(n-2)$-nd secant variety $\cV_{n-1,2}(q)^{(n-1)}$ of the Veronese variety $\cV_{n-1,d}(q)$. 
If a symplectic semifield subspace is not properly contained in another symplectic semifield subspace, we say that it is {\it maximal}. An $(n-1)$-dimensional symplectic semifield subspace $\PG(N-1,q)$ is called a {\it maximum symplectic semifield subspace}.
\end{definition}

The natural definition of equivalence in this setting is equivalence under the group $K\cong \PGL(n,q)$ as a subgroup of $\PGL(N,q)$ stabilising the Veronese variety. The action of $K$ on $\PG(N-1,q)$ is induced by the action of $\GL(n,q)$ on symmetric matrices given by $A\mapsto XAX^T$, and can also be seen as the lift of the action of $\PGL(n,q)$ on $\PG(n-1,q)$ through the Veronese map. In terms of semifields, this corresponds to the notion of {\it strong isotopy}. We note that the more general definition of equivalence of semifields is using {\it isotopy}; this corresponds to equivalence under the action of a subgroup of $\PGL(n^2,q)$ stablising a {\it Segre variety} in $\PG(n^2-1,q)$ \cite{Lavrauw2008}. It can occur that symplectic semifields are isotopic but not strongly isotopic; see for example \cite{Zhou2020}, but in the cases considered in this paper, these two notions coincide.

\section{Description of the algoritms and proofs of the main results}
\label{sec:sss}

The computational classification of symplectic 4-dimensional semifields order $q^4$ for $q\leq 9$ was obtained using GAP \cite{gap} and the package FinInG \cite{fining} for computations in Finite Incidence Geometry.

Let $K\cong \PGL(4,q)$ denote the subgroup of $\PGL(10,q)$ obtained as the image of the lifting homomorphism of the action of $\PGL(4,q)$ from $\PG(3,q)$ to $\PG(9,q)$ through the quadratic Veronese map as described above.
As explained in the previous sections, in order to classify the 4-dimensional symplectic semifields up to isotopism, it suffices to classify their Knuth orbits, and this amounts to classifying the $K$-orbits of solids in $\PG(9,q)$,  which are disjoint from the secant variety $\cV_{n,2}(q)^{(3)}$. Such solids are called {\it symplectic semifield solids}.

Since the problem is equivalent to classifying the $K$-orbits of symplectic semifield solids in $\PG(9,q)$, i.e. solids which are disjoint from the secant variety $\cV_{n,2}(q)^{(3)}$, and there are about $10^{22}/2$ solids in $\PG(9,8)$ and close to $10^{23}$ solids in $\PG(9,9)$ a brute force approach is not feasible. 
The main structure of the algorithm is a breadth first type algorithm with isomorphism rejection under $K$ at every step, based on ideas from the algorithm {\it snakes and ladders} as described in \cite[Section 9.6]{BeBrFrKeKoWa2006}. 
We also used ideas from the orbit computation functionality provided by the FinInG package \cite{fining} and the Orb package \cite{orb}, whose methods we adjusted and reimplemented, tailored for the specific problem at hand, in order to make the computations feasible.

In order to compute the $K$-orbits on symplectic semifield solids our algorithm computes the $K$-orbits of all symplectic semifield subspaces of projective dimension at most three. Moreover, the algorithm collects interesting date along the way, allowing us to recover the stabiliser groups and maximality of the symplectic semifield subspaces.

We first computed the orbits of $K$ on points of $\PG(9,q)$, and for each point in each orbit an element of $K$ mapping that point to the representative of its orbit. There are two orbits of points of rank four in $\PG(9,q)$ with representatives $p_1(1,0,0,0, 1,0,0, 1,0, 1)$ and $p_2(1,0,0,0, 1,0,0, 1,0, \omega)$, where $\omega$ is a primitive element in $\bF_q$. 

For each of the representatives $p_i$ and $i=1,2$ of the $K$-orbits on points of $\PG(9,q)$, we compute its stabiliser $K_{p_i}$ in $K$ and the orbits of that stabiliser on the symplectic semifield lines through $p_i$. 

For each dimension $d\leq 2$, the algorithm stores a list of orbit representatives of $d$-dimensional symplectic semifield subspaces of $\PG(9,q)$, together with a Schreier vector (for constructive recognition) and the stabliser of the representative. At the next step, for each representative $W$, the algorithm computes the orbits on the set of all possible extensions (to $(d+1)$-dimensional symplectic semifield subspaces) of $W$ under the stabiliser of $W$. This gives a list of all $(d+1)$-dimensional symplectic semifield subspaces containing a representative of the $d$-dimensional symplectic semifield subspaces. The list of $(d+1)$-dimensional symplectic semifield subspaces is then reduced (using the stored Schreier vectors) by removing the representatives from the list which belong to the same $K$-orbit as a representative which comes earlier in the list.

The number of symplectic semifield subspaces and of maximal symplectic subspaces in $\PG(9,q)$, for $q\leq 9$, up to $K$-equivalence, are listed in the following two tables below.

\begin{table}[htp]
\caption{Symplectic $d$-dimensional semifield subspaces in $\PG(9,q)$}
\begin{center}
\begin{tabular}{c|c|c|c|c|c|c|c|}
$q$ 		& 2& 3& 4 & 5 & 7 & 8 & 9\\
\hline
 $d=0$	& 2& 2& 2 & 2 & 2 & 2 & 2\\
\hline
 $d=1$	& 5& 7& 9 & 13 & 16 & 17 & 22\\
\hline
 $d=2$	& 6& 18& 31 & 59 & 106 &100 & 149\\
 \hline
 $d=3$	& 1& 2& 1 & 2 & 2 & 1 & 2 \\
\end{tabular}
\end{center}
\label{tab:sem_subs}
\end{table}%
For completeness, we include the values for $d=0$ in Table \ref{tab:sem_subs}, but note that this follows from the well-known classification of non-singular quadrics in $\PG(3,q)$.
The values in Table \ref{tab:sem_subs} in the last row for $q=2,3,4,5,7$ confirm previous results mentioned in the introduction. The other values in Table \ref{tab:sem_subs} are new. The values in the last row give the number of $K$-orbits of symplectic solids and lead to the following classification results.

\begin{theorem}
There exist no non-associative 4-dimensional symplectic semifields over $\bF_q$, $q$ even, $q\leq 8$.
\end{theorem}
\begin{proof}
The values in the last row of Table \ref{tab:sem_subs} for $q=2,4,8$ indicate that $K$ acts transitively on the set of semifield solids in $\PG(9,q)$. Let $q\in \{2,4,8\}$. Since the finite field $\bF_{q^4}$, when considered as a 4-dimensional algebra ${\mathbb{A}}$ over $\bF_q$, gives rise to a semifield solid $U$ in $\PG(9,q)$, the unique $K$-orbit of semifield solids equals the orbit of $U$ under $K$. This implies that each 4-dimensional symplectic semifield over $\bF_q$ is isotopic to $\mathbb A$, and is therefore associative.
\end{proof}

\begin{theorem}
A non-associative 4-dimensional symplectic semifield over $\bF_q$, $q$ odd, $q\leq 9$ is Knuth-equivalent to a Dickson commutative semifield.
\end{theorem}
\begin{proof}
The Dickson commutative semifield, which we denote by $\bS_D$, is defined by the binary operation
$$
(x,y)\circ(u,v)=(xu+\eta(yv)^\sigma,xv+yu)
$$
on $\bF_{q^k}\times\bF_{q^k}$, where $q$ is odd, $\sigma \in {\mathrm{Gal}}(\bF_{q^k},\bF_q)$, and $\eta\in \bF_{q^k}$ is a non-square. The isotopism class $[\bS_D]^{td}$, which is the transpose dual of the isotopism class $[\bS_D]$, belongs to the Knuth orbit of $\bS_D$ and is symplectic. Thus if $k$=2, then $[\bS_D]^{td}$ corresponds to a $K$-orbit of 3-dimensional symplectic semifield subspaces in $\PG(9,q)$. Moreover, since $\bS_D^{td}$ is non-associative this $K$-orbit is distinct from the $K$-orbit of 3-dimensional semifield subspaces obtained from the finite field $\bF_{q^4}$, considered as a 4-dimensional algebra over $\bF_q$. Since the value for $q=9$ in the last row of Table \ref{tab:sem_subs} is equal to two, there is a unique such $K$-orbit in $\PG(9,9)$. Hence if $\bS$ is any non-associative 4-dimensional symplectic semifield over $\bF_9$, then the associated 3-dimensional symplectic semifield subspace belongs to this $K$-orbit, and therefore $\bS$ is isotopic to $\bS_D^{td}$. This concludes the proof.
\end{proof}

Since the algorithm keeps track of all the $K$-orbits of intermediate symplectic subspaces, we can also recover those symplectic subspaces which are not extendible to a larger symplectic subspace, i.e. the maximal ones. For instance, out of the 59 representatives for the $K$-orbits on symplectic 2-dimensional semifield subspaces in $\PG(9,5)$ all but 10 are maximal. For $q=7$ there are 90 $K$-orbits on maximal symplectic 2-dimensional semifield subspaces in $\PG(9,7)$. We summarize this in the Table \ref{tab:maximals} for the values of $q\leq 9$.

\begin{table}[htp]
\caption{The number of isotopism classes of maximal symplectic semifield subspaces in $\PG(9,q)$}
\begin{center}
\begin{tabular}{c|c|c|c|c|c|c|c|}
$q$ 		&2 & 3 &4 & 5 & 7&8& 9\\
\hline
 $d=0$	&0 & 0 &0 & 0 & 0 & 0 & 0\\
\hline
 $d=1$	&1 & 0 &1 & 0 & 0 & 1 & 0\\
\hline
 $d=2$	&5 & 13 &30 & 49 & 90 & 99 & 124\\
 \hline
 $d=3$	&1 & 2 &1 & 2 & 2 & 1 & 2 \\
\end{tabular}
\end{center}
\label{tab:maximals}
\end{table}%

\section{Symplectic semifield lines}\label{sec:sslines}

As before, let $K$ denote the lift of $\PGL(4,q)$ through the Veronese map, i.e. 
$K\cong \PGL(4,q)$, $K\leq \PGL(10,q)$, and consider its action on subspaces of $\PG(9,q)$.
In this section we list the representatives for the $K$-orbits of $d$-dimensional symplectic semifield subspaces in $\PG(9,q)$, for $d\in \{1,2\}$, and $q\in\{8,9\}$. For simplicity and clarity these subspaces are represented by symmetric $3\times 3$-matrices where the entries below the diagonal are omitted, and where the symbols $x,y,z$ are used as parameters.

\subsection{Symplectic semifield lines in $\PG(9,8)$}

There are 17 one-dimensional symplectic semifield subspaces of $\PG(9,8)$ up to $K$-equivalence, where $K\cong \PGL(4,8)$, $K\leq \PGL(10,8)$. We list a representative for each of the 17 $K$-orbits. The parameters $(x,y)$ run over the elements of $\PG(1,8)$ and $\alpha$ is a primitive element of $\bF_8$ with minimal polynomial $X^3+X+1 \in \bF_2[X]$.

%\bigskip

%\newpage

{\begin{tiny}
\begin{center}
\begin{longtable}{lll}
1. $\semsymmat{x}{y}{0}{0}{x}{0}{y}{x}{y}{x}$
&
2. $\semsymmat{x}{y}{0}{0}{x}{0}{y}{x}{y}{x + y}$
&
3. $\semsymmat{x}{y}{0}{0}{x}{0}{y}{x}{y}{x + \alpha y}$
\\
\\
4. $\semsymmat{x}{y}{0}{0}{x}{0}{y}{x}{y}{x + \alpha^2 y}$
&
5. $\semsymmat{x}{y}{0}{0}{x}{0}{y}{x}{y}{x + \alpha^4 y}$
&
6. $\semsymmat{x}{y}{0}{0}{x}{0}{y}{x}{\alpha y}{x + y}$
\\
\\
7. $\semsymmat{x}{y}{0}{0}{x}{0}{y}{x}{\alpha y}{x + \alpha^4 y}$
&
8. $\semsymmat{x}{y}{0}{0}{x}{0}{y}{x}{\alpha^2 y}{x + y}$
&
9. $\semsymmat{x}{y}{0}{0}{x + y}{0}{0}{x + y}{y}{x}$
\\
\\
10. $\semsymmat{x}{y}{0}{0}{x + y}{0}{\alpha y}{x + y}{\alpha^3 y}{x}$
&
11. $\semsymmat{x}{y}{0}{0}{x + y}{0}{\alpha y}{x + \alpha y}{\alpha y}{x +
\alpha^3 y}$
&
12. $\semsymmat{x}{y}{0}{0}{x + y}{0}{\alpha^2 y}{x + y}{\alpha^6 y}{x}$
\\
\\
13. $\semsymmat{x}{y}{0}{0}{x + y}{0}{\alpha^2 y}{x + \alpha^4 y}{\alpha^6 y}{x +
\alpha^5 y}$
&
14. $\semsymmat{x}{y}{0}{0}{x + y}{0}{\alpha^3 y}{x + \alpha^6 y}{\alpha y}{x +
\alpha^2 y}$
&
15. $\semsymmat{x}{y}{0}{0}{x + y}{0}{\alpha^4 y}{x + \alpha y}{\alpha^5 y}{x +
\alpha^3 y}$
\\
\\
16. $\semsymmat{0}{x}{y}{0}{0}{0}{y}{0}{x + y}{0}$
&
17. $\semsymmat{x}{y}{0}{0}{x + \alpha^3 y}{0}{y}{x + \alpha^6 y}{y}{x + \alpha^4 y}$
& 
\end{longtable}
\end{center}

\end{tiny}}

\subsection{Symplectic semifield lines in $\PG(9,9)$}
There are 22 $K$-orbits of symplectic semifield subspaces of dimension one in $\PG(9,9)$, where $K\cong \PGL(4,9)$, $K\leq \PGL(10,9)$. For each of the $K$-orbits we list a representative. The parameters $(x,y)$ run over the elements of $\PG(1,9)$ and $\alpha$ is a primitive element of $\bF_9$ with minimal polynomial $X^2-X-1 \in \bF_3[X]$.

\newpage

\bigskip

{\begin{tiny}
\begin{center}
\begin{longtable}{lll}
1. $\semsymmat{x}{y}{0}{0}{x}{0}{y}{x}{y}{x + y}$ & 
2. $\semsymmat{x}{y}{0}{0}{x}{0}{y}{x}{y}{x + \alpha y} $ &
3. $\semsymmat{x}{y}{0}{0}{x}{0}{y}{x}{y}{x + \alpha^3 y}$
\\
\\
4. $\semsymmat{x}{y}{0}{0}{x}{0}{y}{x}{\alpha y}{x}$ 
&
5. $ \semsymmat{x}{y}{0}{0}{x}{0}{y}{x}{\alpha y}{x + \alpha y}
 $ & 
6. $ \semsymmat{x}{y}{0}{0}{x}{0}{y}{x}{\alpha y}{x + \alpha^3 y}
$ 
\\
\\
7. $
\semsymmat{x}{y}{0}{0}{x}{0}{y}{x}{\alpha^2 y}{x}
 $ &
 8. $ 
\semsymmat{x}{y}{0}{0}{x}{0}{y}{x + \alpha y}{y}{x + \alpha y}
 $
 &
9. $ 
\semsymmat{x}{y}{0}{0}{x}{0}{y}{x + \alpha y}{y}{x + \alpha^3 y}
$
\\
\\
10. $
\semsymmat{x}{y}{0}{0}{x}{0}{y}{x + \alpha y}{y}{x + \alpha^6 y}
 $ & 
 11. $ 
\semsymmat{x}{y}{0}{0}{x}{0}{y}{x + \alpha y}{\alpha^2 y}{x + \alpha^2 y}
 $ &
 12. $ 
\semsymmat{x}{y}{0}{0}{x}{0}{y}{x + \alpha y}{\alpha^2 y}{x}
$ 
\\ 
\\
13. $
\semsymmat{x}{y}{0}{0}{x}{0}{y}{x + \alpha y}{\alpha^2 y}{x + \alpha^6 y}
 $ & 
 14. $ 
\semsymmat{x}{y}{0}{0}{x}{0}{y}{x + \alpha y}{\alpha^3 y}{x + \alpha^6 y}
 $ & 
 15. $ 
\semsymmat{x}{y}{0}{0}{x}{0}{y}{x + \alpha^2 y}{y}{x + \alpha^3 y}
$ 
\\
\\
16. $
\semsymmat{x}{y}{0}{0}{x}{0}{y}{x + \alpha^2 y}{y}{x}
 $ 
 &
17. $ 
\semsymmat{x}{y}{0}{0}{x}{0}{y}{x + \alpha^2 y}{y}{x + y}
 $ & 
 18. $ 
\semsymmat{x}{y}{0}{0}{x}{0}{\alpha y}{x}{y}{x}
$
\\
\\
19. $
\semsymmat{x}{y}{0}{0}{x}{0}{\alpha y}{x + \alpha^3 y}{\alpha^2 y}{x +
\alpha^2 y}
 $ & 
 20. $ 
\semsymmat{x}{y}{0}{0}{x}{y}{\alpha y}{x + \alpha^2 y}{\alpha^6 y}{x +
\alpha^5 y}
 $ 
 &
21. $ 
\semsymmat{x}{y}{0}{0}{x + \alpha y}{0}{0}{x + \alpha y}{y}{x}
$
\\
\\
22. $
\semsymmat{x}{y}{0}{0}{x + \alpha y}{0}{0}{x + \alpha^2 y}{\alpha^3 y}{\alpha x
+ \alpha^5 y}$
  &  &
\end{longtable}
\end{center}

\end{tiny}

\section{Symplectic semifield subspace of dimension two}\label{sec:ssplanes}

\subsection{Symplectic two-dimensional semifield subspaces of $\PG(9,8)$}

The following is a list of representatives of the $K$-orbits of two-dimensional symplectic semifield subspaces of $\PG(9,8)$, where $K\cong \PGL(4,8)$, $K\leq \PGL(10,8)$.
The parameters $(x,y,z)$ run over the elements of $\PG(2,8)$ and $\alpha$ is a primitive element of $\bF_8$ with minimal polynomial $X^3+X+1 \in \bF_2[X]$.

%\newpage

\begin{tiny}
\begin{center}
\begin{longtable}{lll}
1. $\semsymmat{x}{y}{z}{\alpha z}{x + \alpha^2 z}{\alpha z}{y + \alpha^2 z}{x +
\alpha^2 z}{y}{x}$
&
2. $\semsymmat{x}{y}{z}{\alpha^4 z}{x + \alpha^2 z}{\alpha^4 z}{y + \alpha^2 z}{x +
\alpha^2 z}{y}{x}$
&
3. $\semsymmat{x}{y}{z}{0}{x + \alpha^4 z}{\alpha^4 z}{y + \alpha^2 z}{x}{y}{x +
\alpha^2 z}$
%\end{tabular}
%\end{center}

\\ \\
4. $
\semsymmat{x}{y}{z}{\alpha^2 z}{x + \alpha^4 z}{\alpha^4 z}{y}{x}{y +
\alpha^2 z}{x + \alpha^2 z}
$ & 5. $
\semsymmat{x}{y}{z}{\alpha z}{x + \alpha z}{\alpha^4 z}{y + \alpha^2 z}{x +
\alpha^2 z}{y + \alpha^2 z}{x}
$ & 6. $
\semsymmat{x}{y}{z}{\alpha^4 z}{x + \alpha z}{\alpha^4 z}{y}{x + \alpha^2 z}{y +
\alpha^2 z}{x + \alpha^2 z}
$
\\ \\

7. $
\semsymmat{x}{y}{z}{0}{x + \alpha z}{\alpha^2 z}{y + \alpha z}{x + \alpha^2 z}{y
+ \alpha^2 z}{x + \alpha^2 z}
$ & 8. $
\semsymmat{x}{y}{z}{\alpha^4 z}{x + \alpha^4 z}{0}{y + \alpha^4 z}{x +
\alpha^2 z}{y}{x}
$ & 9. $
\semsymmat{x}{y}{z}{\alpha^4 z}{x + \alpha^4 z}{\alpha^4 z}{y + \alpha z}{x +
\alpha^2 z}{y}{x}
$
\\ \\
10. $
\semsymmat{x}{y}{z}{\alpha z}{x + \alpha z}{\alpha z}{y + \alpha z}{x +
\alpha^2 z}{y + \alpha^2 z}{x}
$ & 11. $
\semsymmat{x}{y}{z}{\alpha^2 z}{x + \alpha^2 z}{0}{y + \alpha^2 z}{x +
\alpha z}{y}{x}
$ & 12. $
\semsymmat{x}{y}{z}{\alpha z}{x + \alpha z}{\alpha z}{y + \alpha^2 z}{x +
\alpha^4 z}{y}{x}
$
\\ \\
13. $
\semsymmat{x}{y}{z}{\alpha^2 z}{x + \alpha^2 z}{\alpha^2 z}{y + \alpha^4 z}{x +
\alpha z}{y}{x}
$ & 14. $
\semsymmat{x}{y}{z}{\alpha^2 z}{x + \alpha z}{\alpha^2 z}{y + \alpha z}{x +
\alpha z}{y}{x}
$ & 15. $
\semsymmat{x}{y}{z}{\alpha^2 z}{x + \alpha z}{\alpha^2 z}{y + \alpha^4 z}{x +
\alpha^4 z}{y + \alpha^2 z}{x}
$
\\ \\
16. $
\semsymmat{x}{y}{z}{\alpha^4 z}{x + \alpha z}{\alpha^4 z}{y + \alpha z}{x +
\alpha z}{y}{x}
$ & 17. $
\semsymmat{x}{y}{z}{\alpha^4 z}{x + \alpha^4 z}{0}{y + \alpha^2 z}{x +
\alpha^2 z}{y + \alpha z}{x}
$ & 18. $
\semsymmat{x}{y}{z}{\alpha^2 z}{x + \alpha^2 z}{\alpha^2 z}{y + \alpha^2 z}{x +
\alpha^4 z}{y + \alpha^4 z}{x}
$
\\ \\
19. $
\semsymmat{x}{y}{z}{\alpha z}{x}{0}{y + \alpha^2 z}{x + \alpha z}{y +
\alpha z}{x + \alpha^2 z}
$ & 20. $
\semsymmat{x}{y}{z}{\alpha^2 z}{x}{\alpha^2 z}{y + \alpha z}{x + \alpha z}{y +
\alpha z}{x + \alpha^2 z}
$ & 21. $
\semsymmat{x}{y}{z}{0}{x + \alpha^2 z}{\alpha^4 z}{y + \alpha^2 z}{x +
\alpha^4 z}{y + \alpha^4 z}{x + \alpha^4 z}
$
\\ \\
22. $
\semsymmat{x}{y}{z}{\alpha^6 z}{x}{\alpha^2 z}{y + \alpha^2 z}{x + \alpha^4 z}{y
+ \alpha^4 z}{x + \alpha^2 z}
$ & 23. $
\semsymmat{x}{y}{z}{\alpha^3 z}{x}{\alpha z}{y + \alpha z}{x + \alpha^2 z}{y +
\alpha^2 z}{x + \alpha z}
$ & 24. $
\semsymmat{x}{y}{z}{\alpha^5 z}{x}{\alpha^4 z}{y + \alpha^4 z}{x + \alpha z}{y +
\alpha z}{x + \alpha^4 z}
$
\\ \\
25. $
\semsymmat{x}{y}{z}{z}{x + \alpha^4 z}{z}{y + \alpha^4 z}{x + \alpha z}{y}{x +
\alpha^2 z}
$ & 26. $
\semsymmat{x}{y}{z}{\alpha z}{x + \alpha^3 z}{\alpha^3 z}{y + \alpha^2 z}{x +
\alpha z}{y + \alpha^2 z}{x + \alpha^2 z}
$ & 27. $
\semsymmat{x}{y}{z}{\alpha z}{x + \alpha^5 z}{\alpha^3 z}{y + \alpha z}{x +
\alpha^2 z}{y + \alpha^4 z}{x + \alpha^2 z}
$
\\ \\
28. $
\semsymmat{x}{y}{z}{\alpha z}{x + \alpha^3 z}{\alpha^5 z}{y + \alpha^4 z}{x +
\alpha^2 z}{y + \alpha^4 z}{x + \alpha^2 z}
$ & 29. $
\semsymmat{x}{y}{z}{0}{x + \alpha^2 z}{\alpha^2 z}{y + z}{x}{y}{x}
$ & 30. $
\semsymmat{x}{y}{z}{0}{x + \alpha^4 z}{\alpha z}{y + z}{x + \alpha^2 z}{y}{x}
$
\\ \\
31. $
\semsymmat{x}{y}{z}{0}{x + \alpha z}{\alpha^4 z}{y + z}{x + \alpha^2 z}{y}{x}
$ & 32. $
\semsymmat{x}{y}{z}{0}{x + \alpha^2 z}{\alpha^4 z}{y + z}{x + \alpha z}{y}{x}
$ & 33. $
\semsymmat{x}{y}{z}{0}{x + \alpha^4 z}{0}{y + \alpha z}{x}{y + \alpha^2 z}{x +
y}
$
\\ \\
34. $
\semsymmat{x}{y}{z}{\alpha^2 z}{x + \alpha^2 z}{\alpha^4 z}{y}{x + \alpha^4 z}{y
+ \alpha^2 z}{x + y + \alpha^2 z}
$ & 35. $
\semsymmat{x}{y}{z}{0}{x + \alpha z}{0}{y + \alpha^2 z}{x}{y + \alpha^4 z}{x +
y}
$ & 36. $
\semsymmat{x}{y}{z}{0}{x + \alpha^2 z}{0}{y + \alpha^4 z}{x}{y + \alpha z}{x +
y}
$
\\ \\
37. $
\semsymmat{x}{y}{z}{\alpha z}{x + \alpha z}{\alpha^2 z}{y}{x + \alpha^2 z}{y +
\alpha z}{x + y + \alpha z}
$ & 38. $
\semsymmat{x}{y}{z}{\alpha^4 z}{x + \alpha^4 z}{\alpha z}{y}{x + \alpha z}{y +
\alpha^4 z}{x + y + \alpha^4 z}
$ & 39. $
\semsymmat{x}{y}{z}{\alpha^6 z}{x + \alpha^2 z}{\alpha z}{y + \alpha^4 z}{x +
\alpha^4 z}{y + \alpha^2 z}{x + y + \alpha^2 z}
$
\\ \\
%
%
%\end{tabular}
%\end{center}
%\end{tiny}
%
%
%
%
%
%
%\begin{tiny}
%
%\begin{center}
%\begin{tabular}{lll}
%
%
%
40. $
\semsymmat{x}{y}{z}{\alpha^3 z}{x + \alpha z}{\alpha^4 z}{y + \alpha^2 z}{x +
\alpha^2 z}{y + \alpha z}{x + y + \alpha z}
$ & 41. $
\semsymmat{x}{y}{z}{\alpha^5 z}{x + \alpha^4 z}{\alpha^2 z}{y + \alpha z}{x +
\alpha z}{y + \alpha^4 z}{x + y + \alpha^4 z}
$ & 42. $
\semsymmat{x}{y}{z}{z}{x + \alpha^3 z}{\alpha z}{y + \alpha z}{x +
\alpha z}{y}{x + y + \alpha^2 z}
$
\\ \\
43. $
\semsymmat{x}{y}{z}{z}{x + \alpha^6 z}{\alpha^2 z}{y + \alpha^2 z}{x +
\alpha^2 z}{y}{x + y + \alpha^4 z}
$ & 44. $
\semsymmat{x}{y}{z}{z}{x + \alpha^5 z}{\alpha^4 z}{y + \alpha^4 z}{x +
\alpha^4 z}{y}{x + y + \alpha z}
$ & 45. $
\semsymmat{x}{y}{0}{z}{x}{z}{y + \alpha z}{x + \alpha^2 z}{y}{x + y +
\alpha^2 z}
$
\\ \\
46. $
\semsymmat{x}{y}{0}{z}{x}{z}{y + \alpha^2 z}{x + \alpha^4 z}{y}{x + y +
\alpha^4 z}
$ & 47. $
\semsymmat{x}{y}{0}{z}{x}{z}{y + \alpha^4 z}{x + \alpha z}{y}{x + y + \alpha z}
$ & 48. $
\semsymmat{x}{y}{0}{z}{x + \alpha^6 z}{z}{y + \alpha^4 z}{x}{y + \alpha^2 z}{x +
y + \alpha^2 z}
$
\\ \\
49. $
\semsymmat{x}{y}{0}{z}{x + \alpha^3 z}{z}{y + \alpha^2 z}{x}{y + \alpha z}{x + y
+ \alpha z}
$ & 50. $
\semsymmat{x}{y}{0}{z}{x + \alpha^5 z}{z}{y + \alpha z}{x}{y + \alpha^4 z}{x + y
+ \alpha^4 z}
$ & 51. $
\semsymmat{x}{y}{z}{\alpha^3 z}{x + \alpha^6 z}{\alpha^3 z}{y}{x + \alpha^4 z}{y
+ \alpha^2 z}{x + y}
$
\\ \\
52. $
\semsymmat{x}{y}{z}{\alpha^5 z}{x + \alpha^3 z}{\alpha^5 z}{y}{x + \alpha^2 z}{y
+ \alpha z}{x + y}
$ & 53. $
\semsymmat{x}{y}{z}{\alpha^6 z}{x + \alpha^5 z}{\alpha^6 z}{y}{x + \alpha z}{y +
\alpha^4 z}{x + y}
$ & 54. $
\semsymmat{x}{y}{z}{0}{x + \alpha z}{0}{y + \alpha^6 z}{x + \alpha z}{y +
\alpha z}{x + y + \alpha^2 z}
$
\\ \\
55. $
\semsymmat{x}{y}{z}{0}{x + \alpha^2 z}{0}{y + \alpha^5 z}{x + \alpha^2 z}{y +
\alpha^2 z}{x + y + \alpha^4 z}
$ & 56. $
\semsymmat{x}{y}{z}{0}{x + \alpha^4 z}{0}{y + \alpha^3 z}{x + \alpha^4 z}{y +
\alpha^4 z}{x + y + \alpha z}
$ & 57. $
\semsymmat{x}{y}{z}{\alpha^5 z}{x + \alpha^2 z}{0}{y + \alpha^5 z}{x +
\alpha^2 z}{y + \alpha z}{x + y + \alpha^2 z}
$
\\ \\
58. $
\semsymmat{x}{y}{z}{\alpha^6 z}{x + \alpha z}{0}{y + \alpha^3 z}{x + \alpha z}{y
+ \alpha^2 z}{x + y + \alpha^4 z}
$ & 59. $
\semsymmat{x}{y}{z}{\alpha^5 z}{x + \alpha^2 z}{0}{y + \alpha^6 z}{x +
\alpha^2 z}{y + \alpha^4 z}{x + y + \alpha z}
$ & 60. $
\semsymmat{x}{y}{z}{0}{x}{0}{y + \alpha^2 z}{x + z}{y}{x + y + \alpha^2 z}
$
\\ \\
61. $
\semsymmat{x}{y}{z}{\alpha^2 z}{x}{\alpha^2 z}{y + \alpha^2 z}{x + z}{y +
\alpha^2 z}{x + y + \alpha^2 z}
$ & 62. $
\semsymmat{x}{y}{z}{\alpha^2 z}{x + \alpha^4 z}{\alpha^2 z}{y + \alpha^4 z}{x +
z}{y + \alpha^2 z}{x + y}
$ & 63. $
\semsymmat{x}{y}{z}{0}{x + \alpha z}{0}{y + \alpha^4 z}{x + z}{y}{x + y +
\alpha^2 z}
$
\\ \\
64. $
\semsymmat{x}{y}{z}{\alpha z}{x + \alpha^2 z}{\alpha z}{y + \alpha^2 z}{x + z}{y
+ \alpha z}{x + y}
$ & 65. $
\semsymmat{x}{y}{z}{\alpha^4 z}{x + \alpha z}{\alpha^4 z}{y + \alpha z}{x + z}{y
+ \alpha^4 z}{x + y}
$ & 66. $
\semsymmat{x}{y}{z}{\alpha z}{x + \alpha^4 z}{\alpha z}{y + \alpha z}{x + z}{y +
\alpha z}{x + y + \alpha^2 z}
$
\\ \\
67. $
\semsymmat{x}{y}{z}{\alpha z}{x + \alpha z}{\alpha z}{y + \alpha^4 z}{x + z}{y +
\alpha z}{x + y + \alpha^2 z}
$ & 68. $
\semsymmat{x}{y}{z}{0}{x + \alpha^4 z}{0}{y + \alpha^2 z}{x + z}{y}{x + y +
\alpha z}
$ & 69. $
\semsymmat{x}{y}{z}{\alpha^2 z}{x + \alpha z}{\alpha^2 z}{y + \alpha^2 z}{x +
z}{y + \alpha^2 z}{x + y + \alpha^4 z}
$
\\ \\
70. $
\semsymmat{x}{y}{z}{0}{x + \alpha^2 z}{0}{y + \alpha z}{x + z}{y}{x + y +
\alpha^4 z}
$ & 71. $
\semsymmat{x}{y}{z}{0}{x}{0}{y + \alpha^4 z}{x + z}{y}{x + y + \alpha^4 z}
$ & 72. $
\semsymmat{x}{y}{z}{\alpha z}{x}{\alpha z}{y + \alpha z}{x + z}{y + \alpha z}{x
+ y + \alpha z}
$
\\ \\
73. $
\semsymmat{x}{y}{z}{\alpha^4 z}{x + \alpha^2 z}{\alpha^4 z}{y + \alpha^4 z}{x +
z}{y + \alpha^4 z}{x + y + \alpha z}
$ & 74. $
\semsymmat{x}{y}{z}{\alpha^4 z}{x}{\alpha^4 z}{y + \alpha^4 z}{x + z}{y +
\alpha^4 z}{x + y + \alpha^4 z}
$ & 75. $
\semsymmat{x}{y}{z}{0}{x + z}{0}{y + z}{x + z}{y}{x + y}
$
\\ \\
76. $
\semsymmat{x}{y}{z}{0}{x + z}{0}{y + \alpha^6 z}{x + z}{y}{x + y + \alpha^2 z}
$ & 77. $
\semsymmat{x}{y}{z}{\alpha^2 z}{x + z}{\alpha^2 z}{y + \alpha^6 z}{x + z}{y +
\alpha^2 z}{x + y + \alpha^2 z}
$ & 78. $
\semsymmat{x}{y}{z}{\alpha^2 z}{x + \alpha^5 z}{\alpha^2 z}{y + \alpha^3 z}{x +
z}{y + \alpha^2 z}{x + y + \alpha^2 z}
$
\\ \\
79. $
\semsymmat{x}{y}{z}{\alpha z}{x + \alpha^5 z}{\alpha z}{y + \alpha^3 z}{x + z}{y
+ \alpha z}{x + y + \alpha^2 z}
$ & 80. $
\semsymmat{x}{y}{z}{0}{x + z}{0}{y + \alpha^3 z}{x + z}{y}{x + y + \alpha z}
$ & 81. $
\semsymmat{x}{y}{z}{0}{x + z}{0}{y + \alpha^5 z}{x + z}{y}{x + y + \alpha^4 z}
$
\\ \\
82. $
\semsymmat{x}{y}{z}{\alpha^4 z}{x + \alpha^3 z}{\alpha^4 z}{y + \alpha^6 z}{x +
z}{y + \alpha^4 z}{x + y + \alpha^4 z}
$ & 83. $
\semsymmat{x}{y}{z}{\alpha z}{x + \alpha^6 z}{\alpha z}{y + \alpha^5 z}{x + z}{y
+ \alpha z}{x + y + \alpha z}
$ & 84. $
\semsymmat{x}{y}{z}{\alpha^4 z}{x + \alpha^6 z}{\alpha^4 z}{y + \alpha^5 z}{x +
z}{y + \alpha^4 z}{x + y + \alpha z}
$
\\ \\
85. $
\semsymmat{x}{y}{z}{\alpha^3 z}{x}{\alpha^6 z}{y + \alpha^4 z}{x + \alpha z}{y +
\alpha^6 z}{x + y + \alpha^2 z}
$ & 86. $
\semsymmat{x}{y}{z}{\alpha^6 z}{x}{\alpha^5 z}{y + \alpha z}{x + \alpha^2 z}{y +
\alpha^5 z}{x + y + \alpha^4 z}
$ & 87. $
\semsymmat{x}{y}{z}{\alpha^5 z}{x}{\alpha^3 z}{y + \alpha^2 z}{x + \alpha^4 z}{y
+ \alpha^3 z}{x + y + \alpha z}
$
\\ \\
88. $
\semsymmat{x}{y}{0}{z}{x + z}{z}{y + z}{x}{y + z}{x + y}
$ & 89. $
\semsymmat{x}{y}{z}{\alpha^6 z}{x + \alpha^2 z}{\alpha^6 z}{y + \alpha^5 z}{x +
\alpha^5 z}{y + z}{x + y + \alpha^2 z}
$ & 90. $
\semsymmat{x}{y}{z}{\alpha^3 z}{x + \alpha z}{\alpha^3 z}{y + \alpha^6 z}{x +
\alpha^6 z}{y + z}{x + y + \alpha z}
$
\\ \\
91. $
\semsymmat{x}{y}{z}{\alpha^5 z}{x + \alpha^4 z}{\alpha^5 z}{y + \alpha^3 z}{x +
\alpha^3 z}{y + z}{x + y + \alpha^4 z}
$ & 92. $
\semsymmat{x}{y}{z}{\alpha^5 z}{x + \alpha^5 z}{\alpha^5 z}{y + \alpha^3 z}{x +
z}{y + \alpha^5 z}{x + y + \alpha^2 z}
$ & 93. $
\semsymmat{x}{y}{z}{\alpha^6 z}{x + \alpha^6 z}{\alpha^6 z}{y + \alpha^5 z}{x +
z}{y + \alpha^6 z}{x + y + \alpha z}
$
\\ \\
94. $
\semsymmat{x}{y}{z}{\alpha^3 z}{x + \alpha^3 z}{\alpha^3 z}{y + \alpha^6 z}{x +
z}{y + \alpha^3 z}{x + y + \alpha^4 z}
$ & 95. $
\semsymmat{x}{y}{z}{0}{x + \alpha^6 z}{0}{y + \alpha^2 z}{x + z}{y}{x + y + z}
$ & 96. $
\semsymmat{x}{y}{z}{0}{x + \alpha^3 z}{0}{y + \alpha z}{x + z}{y}{x + y + z}
$
\\ \\
%
%\end{tabular}
%\end{center}
%
%
%\begin{center}
%\begin{tabular}{lll}
%
%
97. $
\semsymmat{x}{y}{z}{0}{x + \alpha^5 z}{0}{y + \alpha^4 z}{x + z}{y}{x + y + z}
$ & 98. $
\semsymmat{x}{y}{z}{\alpha^6 z}{x + z}{\alpha z}{y + \alpha z}{x + \alpha z}{y +
\alpha^2 z}{x + \alpha y + \alpha^4 z}
$ & 99. $
\semsymmat{x}{y}{z}{\alpha z}{x + \alpha^6 z}{\alpha^6 z}{y + \alpha^5 z}{x +
z}{y + \alpha^6 z}{x + \alpha y + \alpha^6 z}
$
\\ \\
100. $
\semsymmat{x}{y}{z}{\alpha^5 z}{x + z}{\alpha^2 z}{y + \alpha^2 z}{x +
\alpha^2 z}{y + \alpha^4 z}{x + \alpha^2 y + \alpha z}
$ &  \\
\end{longtable}
\end{center}

\end{tiny}

%\clearpage

\subsection{Two-dimensional symplectic semifield subspaces in $\PG(9,9)$}

The following is a list of representatives of the $K$-orbits of two-dimensional symplectic semifield subspaces of $\PG(9,9)$, where $K\cong \PGL(4,9)$, $K\leq \PGL(10,9)$.
The parameters $(x,y,z)$ run over the elements of $\PG(2,9)$ and $\alpha$ is a primitive element of $\bF_9$ with minimal polynomial $X^3+X+1 \in \bF_3[X]$.

\begin{tiny}
%\clearpage
\begin{center}
\begin{longtable}{lll}
1. $ \semsymmat{x}{y}{0}{z}{x + \alpha^5 z}{\alpha z}{y + \alpha^5 z}{x}{y +
\alpha z}{x + y + \alpha^5 z}
$ &
2. $ \semsymmat{x}{y}{0}{z}{x + \alpha^3 z}{\alpha z}{y + \alpha^5 z}{x}{y +
\alpha^5 z}{x + y + \alpha z}
$ &
3. $ \semsymmat{x}{y}{z}{z}{x + \alpha^2 z}{0}{y + \alpha z}{x}{y + \alpha z}{x + y}
$ \\ \\
4. $ \semsymmat{x}{y}{z}{\alpha^2 z}{x + \alpha^3 z}{\alpha z}{y + \alpha^5 z}{x}{y +
\alpha^5 z}{x + y + \alpha^5 z}
$ &
5. $ \semsymmat{x}{y}{z}{\alpha^7 z}{x + \alpha^2 z}{\alpha^5 z}{y}{x + \alpha z}{y +
\alpha z}{x + y + \alpha^5 z}
$ &
6. $ \semsymmat{x}{y}{z}{0}{x + \alpha^6 z}{\alpha z}{y + \alpha z}{x}{y}{x + y +
\alpha^5 z}
$ \\ \\
7. $ \semsymmat{x}{y}{0}{z}{x + \alpha^7 z}{\alpha z}{y + \alpha^5 z}{x}{y +
\alpha^5 z}{x + y}
$ &
8. $ \semsymmat{x}{y}{0}{z}{x + 2 z}{\alpha z}{y + \alpha z}{x}{y}{x + y + \alpha z}
$ &
9. $ \semsymmat{x}{y}{z}{0}{x}{\alpha^3 z}{y + \alpha^5 z}{x + \alpha z}{y}{x + y +
\alpha^5 z}
$ \\ \\
10. $ \semsymmat{x}{y}{z}{\alpha^7 z}{x + \alpha^5 z}{\alpha^3 z}{y + \alpha z}{x +
\alpha z}{y + \alpha^5 z}{x + y + \alpha^5 z}
$ &
11. $ \semsymmat{x}{y}{z}{0}{x + \alpha^2 z}{z}{y + \alpha^5 z}{x + \alpha z}{y}{x +
y}
$ &
12. $ \semsymmat{x}{y}{z}{\alpha^5 z}{x + \alpha^3 z}{z}{y + \alpha^5 z}{x +
\alpha z}{y + \alpha^5 z}{x + y}
$ \\ \\
13. $ \semsymmat{x}{y}{z}{\alpha^5 z}{x + z}{\alpha^3 z}{y + \alpha z}{x + \alpha z}{y
+ \alpha^5 z}{x + y + \alpha z}
$ &
14. $ \semsymmat{x}{y}{z}{z}{x + 2 z}{\alpha^3 z}{y}{x}{y}{x + y + \alpha z}
$ &
15. $ \semsymmat{x}{y}{z}{\alpha^2 z}{x + \alpha z}{\alpha^7 z}{y + \alpha z}{x +
\alpha^5 z}{y + \alpha^5 z}{x + y + \alpha z}
$ \\ \\
16. $ \semsymmat{x}{y}{z}{\alpha^7 z}{x + z}{2 z}{y}{x + \alpha^5 z}{y}{x + y}
$ &
17. $ \semsymmat{x}{y}{z}{2 z}{x + \alpha^3 z}{\alpha^6 z}{y + \alpha^5 z}{x +
\alpha^5 z}{y + \alpha^5 z}{x + y + \alpha z}
$ &
18. $ \semsymmat{x}{y}{0}{z}{x + \alpha^7 z}{2 z}{y + \alpha z}{x}{y + \alpha z}{x + y
+ \alpha z}
$ \\ \\
19. $ \semsymmat{x}{y}{z}{2 z}{x + \alpha^6 z}{\alpha^6 z}{y + \alpha^5 z}{x +
\alpha^5 z}{y + \alpha z}{x + y + \alpha^5 z}
$ &
20. $ \semsymmat{x}{y}{z}{\alpha z}{x + \alpha^5 z}{0}{y + \alpha^3 z}{x + \alpha z}{y
+ \alpha^5 z}{x + y + \alpha^5 z}
$ &
21. $ \semsymmat{x}{y}{z}{\alpha^7 z}{x}{\alpha^5 z}{y + \alpha^3 z}{x +
\alpha z}{y}{x + y}
$ \\ \\
22. $ \semsymmat{x}{y}{z}{\alpha z}{x + z}{\alpha z}{y + \alpha^3 z}{x}{y +
\alpha z}{x + y + \alpha^5 z}
$ &
23. $ \semsymmat{x}{y}{0}{z}{x + \alpha^7 z}{\alpha z}{y + \alpha^2 z}{x}{y +
\alpha z}{x + y + \alpha z}
$ &
24. $ \semsymmat{x}{y}{z}{2 z}{x + \alpha^6 z}{\alpha z}{y + \alpha^2 z}{x +
\alpha z}{y + \alpha z}{x + y}
$ \\ \\
25. $ \semsymmat{x}{y}{0}{z}{x + \alpha z}{z}{y + z}{x}{y + \alpha^5 z}{x + y +
\alpha z}
$ &
26. $ \semsymmat{x}{y}{0}{z}{x + \alpha^5 z}{\alpha^2 z}{y + \alpha^2 z}{x +
\alpha^5 z}{y + \alpha z}{x + y + \alpha^5 z}
$ &
27. $ \semsymmat{x}{y}{z}{\alpha^7 z}{x + \alpha z}{\alpha^2 z}{y + \alpha^2 z}{x +
\alpha^5 z}{y + \alpha z}{x + y + \alpha z}
$ \\ \\
28. $ \semsymmat{x}{y}{z}{\alpha^7 z}{x + \alpha^5 z}{z}{y + \alpha^3 z}{x}{y +
\alpha^5 z}{x + y + \alpha z}
$ &
29. $ \semsymmat{x}{y}{z}{\alpha^3 z}{x + z}{\alpha^3 z}{y + \alpha^2 z}{x}{y +
\alpha^5 z}{x + y + \alpha z}
$ &
30. $ \semsymmat{x}{y}{z}{\alpha^3 z}{x + \alpha^5 z}{\alpha^6 z}{y + \alpha^3 z}{x +
\alpha^5 z}{y + \alpha^5 z}{x + y}
$ \\ \\
31. $ \semsymmat{x}{y}{z}{\alpha^3 z}{x + \alpha z}{2 z}{y + \alpha^2 z}{x}{y +
\alpha z}{x + y + \alpha^5 z}
$ &
32. $ \semsymmat{x}{y}{z}{\alpha^3 z}{x + \alpha z}{0}{y + \alpha^7 z}{x +
\alpha z}{y}{x + y + \alpha^5 z}
$ &
33. $ \semsymmat{x}{y}{0}{z}{x + z}{\alpha z}{y + \alpha^7 z}{x + \alpha z}{y +
\alpha^5 z}{x + y + \alpha z}
$ \\ \\
34. $ \semsymmat{x}{y}{0}{z}{x + \alpha^7 z}{\alpha z}{y + 2 z}{x + \alpha z}{y +
\alpha z}{x + y + \alpha z}
$ &
35. $ \semsymmat{x}{y}{z}{\alpha^2 z}{x + \alpha^7 z}{\alpha^5 z}{y + \alpha^6 z}{x +
\alpha z}{y + \alpha^5 z}{x + y + \alpha z}
$ &
36. $ \semsymmat{x}{y}{z}{2 z}{x + \alpha^6 z}{0}{y + \alpha^7 z}{x}{y}{x + y +
\alpha z}
$ \\ \\
37. $ \semsymmat{x}{y}{z}{\alpha^5 z}{x + \alpha^7 z}{\alpha^6 z}{y + 2 z}{x +
\alpha^5 z}{y + \alpha z}{x + y + \alpha z}
$ &
38. $ \semsymmat{x}{y}{z}{\alpha^6 z}{x + \alpha^5 z}{\alpha^5 z}{y + \alpha^5 z}{x +
z}{y + \alpha^5 z}{x + y}
$ &
39. $ \semsymmat{x}{y}{z}{\alpha z}{x + \alpha^3 z}{\alpha z}{y}{x + z}{y}{x + y}
$ \\ \\
40. $ \semsymmat{x}{y}{z}{0}{x + z}{\alpha z}{y + \alpha^5 z}{x + \alpha^3 z}{y +
\alpha z}{x + y + \alpha z}
$ &
41. $ \semsymmat{x}{y}{0}{z}{x + \alpha z}{\alpha^2 z}{y + \alpha^5 z}{x +
\alpha^2 z}{y}{x + y + \alpha z}
$ &
42. $ \semsymmat{x}{y}{z}{\alpha^3 z}{x + \alpha z}{\alpha^3 z}{y}{x + z}{y}{x + y}
$ \\ \\
43. $ \semsymmat{x}{y}{z}{\alpha^5 z}{x}{\alpha^7 z}{y + \alpha z}{x +
\alpha^3 z}{y}{x + y}
$ &
44. $ \semsymmat{x}{y}{z}{z}{x + \alpha^2 z}{2 z}{y}{x + \alpha^2 z}{y + \alpha^5 z}{x
+ y + \alpha^5 z}
$ &
45. $ \semsymmat{x}{y}{z}{\alpha^7 z}{x + \alpha z}{\alpha z}{y + z}{x + z}{y +
\alpha^5 z}{x + y}
$ \\ \\
%\end{array}
%\]
%%\clearpage
%\[
%\begin{array}{lll}
46. $ \semsymmat{x}{y}{z}{2 z}{x}{0}{y + \alpha^3 z}{x + z}{y + \alpha z}{x + y +
\alpha z}
$ &
47. $ \semsymmat{x}{y}{z}{\alpha^6 z}{x}{0}{y + \alpha^2 z}{x + z}{y + \alpha^5 z}{x +
y + \alpha z}
$ &
48. $ \semsymmat{x}{y}{z}{\alpha^6 z}{x + 2 z}{\alpha^5 z}{y + \alpha^3 z}{x + z}{y}{x
+ y}
$ \\ \\
49. $ \semsymmat{x}{y}{z}{z}{x}{\alpha^3 z}{y + \alpha^2 z}{x + \alpha^3 z}{y +
\alpha^5 z}{x + y}
$ &
50. $ \semsymmat{x}{y}{z}{\alpha^2 z}{x + \alpha^5 z}{\alpha^2 z}{y + \alpha^3 z}{x +
\alpha^2 z}{y}{x + y + \alpha z}
$ &
51. $ \semsymmat{x}{y}{z}{\alpha^2 z}{x}{\alpha^3 z}{y + \alpha^3 z}{x + \alpha^2 z}{y
+ \alpha z}{x + y + \alpha^5 z}
$ \\ \\
52. $ \semsymmat{x}{y}{z}{\alpha^7 z}{x + \alpha^7 z}{\alpha^2 z}{y + z}{x +
\alpha^3 z}{y}{x + y + \alpha z}
$ &
53. $ \semsymmat{x}{y}{z}{\alpha z}{x + \alpha^3 z}{\alpha^6 z}{y + z}{x + z}{y +
\alpha z}{x + y}
$ &
54. $ \semsymmat{x}{y}{z}{0}{x + \alpha^6 z}{z}{y + \alpha^7 z}{x + \alpha^3 z}{y}{x +
y}
$ \\ \\
55. $ \semsymmat{x}{y}{z}{\alpha^5 z}{x + \alpha^3 z}{\alpha^6 z}{y + \alpha^7 z}{x +
z}{y + \alpha^5 z}{x + y}
$ &
56. $ \semsymmat{x}{y}{0}{z}{x + z}{\alpha^7 z}{y + \alpha^6 z}{x + z}{y + \alpha z}{x
+ y + \alpha^5 z}
$ &
57. $ \semsymmat{x}{y}{z}{\alpha^3 z}{x + \alpha^2 z}{\alpha^2 z}{y + \alpha z}{x +
2 z}{y}{x + y + \alpha z}
$ \\ \\
58. $ \semsymmat{x}{y}{0}{z}{x + \alpha z}{\alpha^7 z}{y + \alpha^5 z}{x + 2 z}{y +
\alpha^5 z}{x + y}
$ &
59. $ \semsymmat{x}{y}{z}{\alpha^3 z}{x + z}{\alpha^6 z}{y}{x + 2 z}{y + \alpha z}{x +
y}
$ &
60. $ \semsymmat{x}{y}{z}{\alpha^2 z}{x}{\alpha^5 z}{y + \alpha^2 z}{x + \alpha^6 z}{y
+ \alpha^5 z}{x + y + \alpha z}
$ \\ \\
61. $ \semsymmat{x}{y}{z}{0}{x + 2 z}{0}{y + \alpha^7 z}{x + \alpha^7 z}{y +
\alpha^5 z}{x + y}
$ &
62. $ \semsymmat{x}{y}{z}{2 z}{x + \alpha^6 z}{\alpha^2 z}{y + 2 z}{x + 2 z}{y +
\alpha z}{x + y}
$ &
63. $ \semsymmat{x}{y}{z}{z}{x + \alpha^6 z}{2 z}{y + \alpha^7 z}{x + 2 z}{y}{x + y +
\alpha z}
$ \\ \\
64. $ \semsymmat{x}{y}{z}{\alpha^2 z}{x + \alpha^7 z}{\alpha z}{y + \alpha^2 z}{x +
\alpha^5 z}{y + \alpha^2 z}{x + y}
$ &
65. $ \semsymmat{x}{y}{z}{\alpha^3 z}{x + \alpha^7 z}{\alpha z}{y + \alpha^3 z}{x +
\alpha^5 z}{y + z}{x + y + \alpha z}
$ &
66. $ \semsymmat{x}{y}{z}{\alpha^3 z}{x + \alpha z}{\alpha^2 z}{y + z}{x + z}{y +
\alpha^3 z}{x + y}
$ \\ \\
67. $ \semsymmat{x}{y}{z}{z}{x + \alpha^6 z}{\alpha^2 z}{y + \alpha^2 z}{x + 2 z}{y +
z}{x + y + \alpha^5 z}
$ &
68. $ \semsymmat{x}{y}{z}{2 z}{x + \alpha^2 z}{\alpha^6 z}{y + 2 z}{x + 2 z}{y +
\alpha^3 z}{x + y}
$ &
69. $ \semsymmat{x}{y}{z}{\alpha^6 z}{x + 2 z}{z}{y + \alpha z}{x + \alpha^5 z}{y +
\alpha^6 z}{x + y}
$ \\ \\
70. $ \semsymmat{x}{y}{z}{\alpha^6 z}{x + \alpha z}{2 z}{y + \alpha z}{x}{y + 2 z}{x +
y}
$ &
71. $ \semsymmat{x}{y}{z}{0}{x}{\alpha^7 z}{y + \alpha^6 z}{x + 2 z}{y + \alpha^7 z}{x
+ y + \alpha^5 z}
$ &
72. $ \semsymmat{x}{y}{z}{0}{x + \alpha^6 z}{\alpha^3 z}{y + \alpha^6 z}{x +
\alpha^5 z}{y + \alpha^5 z}{x + y + \alpha^7 z}
$ \\ \\
73. $ \semsymmat{x}{y}{z}{2 z}{x}{\alpha^5 z}{y + \alpha^5 z}{x + \alpha^5 z}{y +
\alpha^5 z}{x + \alpha y + \alpha z}
$ &
74. $ \semsymmat{x}{y}{z}{0}{x + \alpha^2 z}{0}{y + \alpha^5 z}{x + \alpha z}{y}{x +
\alpha y}
$ &
75. $ \semsymmat{x}{y}{z}{0}{x + 2 z}{0}{y + \alpha z}{x + \alpha z}{y}{x + \alpha y +
\alpha z}
$ \\ \\
76. $ \semsymmat{x}{y}{z}{\alpha^2 z}{x + \alpha^7 z}{0}{y + \alpha^5 z}{x +
\alpha z}{y + \alpha z}{x + \alpha y}
$ &
77. $ \semsymmat{x}{y}{z}{\alpha^7 z}{x + \alpha^6 z}{\alpha z}{y}{x + \alpha^5 z}{y +
\alpha^5 z}{x + \alpha y + \alpha^5 z}
$ &
78. $ \semsymmat{x}{y}{z}{\alpha^6 z}{x + \alpha^5 z}{z}{y + \alpha z}{x}{y +
\alpha z}{x + \alpha y + \alpha z}
$ \\ \\
79. $ \semsymmat{x}{y}{z}{z}{x + \alpha^2 z}{z}{y + \alpha^5 z}{x + \alpha z}{y +
\alpha z}{x + \alpha y}
$ &
80. $ \semsymmat{x}{y}{z}{\alpha^7 z}{x + 2 z}{z}{y}{x + \alpha^5 z}{y + \alpha z}{x +
\alpha y + \alpha z}
$ &
81. $ \semsymmat{x}{y}{z}{2 z}{x + \alpha^7 z}{2 z}{y + \alpha z}{x + \alpha z}{y +
\alpha^5 z}{x + \alpha y + \alpha^5 z}
$ \\ \\
82. $ \semsymmat{x}{y}{z}{\alpha z}{x}{\alpha^5 z}{y + \alpha^2 z}{x}{y}{x + \alpha y
+ \alpha z}
$ &
83. $ \semsymmat{x}{y}{z}{0}{x}{0}{y + z}{x + \alpha z}{y}{x + \alpha y + \alpha z}
$ &
84. $ \semsymmat{x}{y}{z}{\alpha^3 z}{x + \alpha^5 z}{\alpha^2 z}{y + \alpha^3 z}{x +
\alpha^5 z}{y + \alpha^5 z}{x + \alpha y}
$ \\ \\
85. $ \semsymmat{x}{y}{z}{\alpha^5 z}{x + z}{\alpha^3 z}{y + \alpha^3 z}{x +
\alpha z}{y}{x + \alpha y + \alpha z}
$ &
86. $ \semsymmat{x}{y}{z}{\alpha^2 z}{x + \alpha^2 z}{2 z}{y + z}{x + \alpha^5 z}{y +
\alpha z}{x + \alpha y}
$ &
87. $ \semsymmat{x}{y}{z}{\alpha^2 z}{x + \alpha^3 z}{2 z}{y + \alpha^3 z}{x +
\alpha^5 z}{y + \alpha^5 z}{x + \alpha y}
$ \\ \\
88. $ \semsymmat{x}{y}{z}{\alpha^5 z}{x + \alpha^7 z}{\alpha^6 z}{y + \alpha^2 z}{x +
\alpha^5 z}{y + \alpha z}{x + \alpha y}
$ &
89. $ \semsymmat{x}{y}{z}{0}{x}{0}{y + 2 z}{x + \alpha z}{y}{x + \alpha y +
\alpha^5 z}
$ &
90. $ \semsymmat{x}{y}{z}{\alpha^6 z}{x}{\alpha^5 z}{y + \alpha^6 z}{x}{y}{x +
\alpha y}
$ \\ \\
%\end{array}
%\]
%
%
%\[
%\begin{array}{lll}
91. $ \semsymmat{x}{y}{z}{z}{x + z}{z}{y + \alpha^6 z}{x + \alpha z}{y + \alpha z}{x +
\alpha y + \alpha z}
$ &
92. $ \semsymmat{x}{y}{z}{\alpha^5 z}{x + \alpha^7 z}{\alpha^2 z}{y + \alpha^7 z}{x +
\alpha^5 z}{y + \alpha z}{x + \alpha y}
$ &
93. $ \semsymmat{x}{y}{z}{2 z}{x + \alpha z}{2 z}{y + 2 z}{x + \alpha z}{y +
\alpha^5 z}{x + \alpha y}
$ \\ \\
94. $ \semsymmat{x}{y}{z}{\alpha z}{x + \alpha z}{\alpha^2 z}{y + z}{x + \alpha^3 z}{y
+ \alpha z}{x + \alpha y}
$ &
95. $ \semsymmat{x}{y}{z}{z}{x + \alpha^3 z}{\alpha^2 z}{y + \alpha^2 z}{x + z}{y +
\alpha^5 z}{x + \alpha y + \alpha^5 z}
$ &
96. $ \semsymmat{x}{y}{z}{2 z}{x + z}{0}{y + \alpha^5 z}{x + 2 z}{y}{x + \alpha y +
\alpha^5 z}
$ \\ \\
97. $ \semsymmat{x}{y}{z}{\alpha^5 z}{x + \alpha z}{\alpha^3 z}{y + \alpha z}{x +
\alpha^7 z}{y + \alpha^5 z}{x + \alpha y}
$ &
98. $ \semsymmat{x}{y}{0}{z}{x + \alpha z}{\alpha^6 z}{y}{x + \alpha^6 z}{y +
\alpha z}{x + \alpha y + \alpha^5 z}
$ &
99. $ \semsymmat{x}{y}{0}{z}{x + \alpha z}{\alpha^5 z}{y + \alpha^2 z}{x +
\alpha^7 z}{y}{x + \alpha y + \alpha^5 z}
$ \\ \\
100. $ \semsymmat{x}{y}{z}{2 z}{x + 2 z}{\alpha^5 z}{y + \alpha^2 z}{x + \alpha^7 z}{y
+ \alpha^5 z}{x + \alpha y + \alpha^5 z}
$ &
101. $ \semsymmat{x}{y}{z}{z}{x + \alpha^7 z}{\alpha^6 z}{y + \alpha^2 z}{x +
\alpha^7 z}{y}{x + \alpha y + \alpha z}
$ &
102. $ \semsymmat{x}{y}{z}{\alpha^6 z}{x + \alpha^6 z}{0}{y + \alpha^6 z}{x + 2 z}{y +
\alpha z}{x + \alpha y + \alpha z}
$ \\ \\
103. $ \semsymmat{x}{y}{z}{\alpha^3 z}{x + \alpha^5 z}{\alpha^5 z}{y + \alpha z}{x +
\alpha^5 z}{y + \alpha^2 z}{x + \alpha y}
$ &
104. $ \semsymmat{x}{y}{z}{\alpha^7 z}{x}{\alpha z}{y + \alpha z}{x + \alpha z}{y +
z}{x + \alpha y + \alpha^5 z}
$ &
105. $ \semsymmat{x}{y}{z}{2 z}{x + \alpha^3 z}{0}{y + \alpha z}{x + \alpha^5 z}{y +
\alpha^2 z}{x + \alpha y + \alpha^5 z}
$ \\ \\
106. $ \semsymmat{x}{y}{z}{\alpha^2 z}{x + 2 z}{2 z}{y + \alpha^5 z}{x}{y +
\alpha^2 z}{x + \alpha y + \alpha^5 z}
$ &
107. $ \semsymmat{x}{y}{z}{\alpha^7 z}{x + 2 z}{\alpha^7 z}{y + \alpha z}{x +
\alpha z}{y + z}{x + \alpha y + \alpha z}
$ &
108. $ \semsymmat{x}{y}{z}{\alpha^5 z}{x + \alpha z}{0}{y + \alpha^2 z}{x +
\alpha^5 z}{y + \alpha^3 z}{x + \alpha y + \alpha z}
$ \\ \\
108. $ \semsymmat{x}{y}{z}{\alpha z}{x + \alpha^7 z}{\alpha z}{y + \alpha^2 z}{x +
\alpha z}{y + \alpha^2 z}{x + \alpha y}
$ &
109. $ \semsymmat{x}{y}{0}{z}{x + \alpha^6 z}{z}{y + z}{x + \alpha z}{y + z}{x +
\alpha y + \alpha z}
$ &
110. $ \semsymmat{x}{y}{z}{\alpha z}{x + \alpha^2 z}{\alpha z}{y + \alpha^6 z}{x +
\alpha z}{y + \alpha^2 z}{x + \alpha y + \alpha^5 z}
$ \\ \\
111. $ \semsymmat{x}{y}{z}{\alpha^2 z}{x + \alpha^5 z}{\alpha^2 z}{y + 2 z}{x +
\alpha z}{y + \alpha^3 z}{x + \alpha y + \alpha z}
$ &
112. $ \semsymmat{x}{y}{z}{\alpha^2 z}{x + \alpha^6 z}{z}{y + 2 z}{x + \alpha^5 z}{y +
\alpha^2 z}{x + \alpha y + \alpha^5 z}
$ &
113. $ \semsymmat{x}{y}{z}{2 z}{x + \alpha^5 z}{\alpha^6 z}{y + \alpha^6 z}{x +
\alpha^5 z}{y + \alpha^2 z}{x + \alpha y + \alpha z}
$ \\ \\
114. $ \semsymmat{x}{y}{z}{\alpha^7 z}{x + \alpha^7 z}{\alpha^7 z}{y + \alpha^7 z}{x +
\alpha z}{y + z}{x + \alpha y + \alpha z}
$ &
115. $ \semsymmat{x}{y}{z}{\alpha^3 z}{x + \alpha^3 z}{\alpha^7 z}{y + z}{x +
\alpha^3 z}{y + \alpha^2 z}{x + \alpha y + \alpha z}
$ &
116. $ \semsymmat{x}{y}{z}{\alpha^2 z}{x + \alpha^5 z}{\alpha^2 z}{y + \alpha^7 z}{x +
2 z}{y + \alpha^3 z}{x + \alpha y + \alpha z}
$ \\ \\
117. $ \semsymmat{x}{y}{z}{\alpha^5 z}{x}{\alpha^5 z}{y}{x + \alpha z}{y +
\alpha^6 z}{x + \alpha y}
$ &
118. $ \semsymmat{x}{y}{0}{z}{x + z}{\alpha^2 z}{y}{x + \alpha^5 z}{y + \alpha^6 z}{x +
\alpha y + \alpha^5 z}
$ &
119. $ \semsymmat{x}{y}{z}{\alpha^6 z}{x}{\alpha^6 z}{y + \alpha^5 z}{x}{y +
\alpha^6 z}{x + \alpha y}
$ \\ \\
120. $ \semsymmat{x}{y}{z}{\alpha^5 z}{x + \alpha^6 z}{\alpha^7 z}{y}{x + \alpha z}{y +
\alpha^7 z}{x + \alpha y}
$ &
121. $ \semsymmat{x}{y}{z}{\alpha^6 z}{x + \alpha^6 z}{\alpha^6 z}{y + \alpha z}{x +
\alpha z}{y + \alpha^7 z}{x + \alpha y}
$ &
122. $ \semsymmat{x}{y}{z}{\alpha^5 z}{x + 2 z}{\alpha^5 z}{y + \alpha^2 z}{x +
\alpha z}{y + \alpha^6 z}{x + \alpha y + \alpha^5 z}
$ \\ \\
123. $ \semsymmat{x}{y}{z}{\alpha^3 z}{x + \alpha z}{\alpha^3 z}{y + z}{x + \alpha z}{y
+ 2 z}{x + \alpha y + \alpha^5 z}
$ &
124. $ \semsymmat{x}{y}{z}{\alpha^5 z}{x + \alpha z}{0}{y + \alpha^5 z}{x +
\alpha^7 z}{y + \alpha^7 z}{x + \alpha y + \alpha z}
$ &
125. $ \semsymmat{x}{y}{z}{0}{x + z}{0}{y + \alpha z}{x + \alpha^7 z}{y + \alpha^6 z}{x
+ \alpha y + \alpha^5 z}
$ \\ \\
126. $ \semsymmat{x}{y}{z}{0}{x + z}{0}{y + 2 z}{x + \alpha z}{y}{x + \alpha y +
\alpha^3 z}
$ &
127. $ \semsymmat{x}{y}{z}{0}{x}{0}{y + z}{x + \alpha^3 z}{y}{x + \alpha y +
\alpha^3 z}
$ &
128. $ \semsymmat{x}{y}{z}{\alpha^7 z}{x + \alpha^7 z}{\alpha^6 z}{y + \alpha^5 z}{x +
\alpha^7 z}{y}{x + \alpha y + \alpha^3 z}
$ \\ \\
129. $ \semsymmat{x}{y}{z}{\alpha z}{x + \alpha^5 z}{\alpha z}{y + \alpha^7 z}{x +
\alpha z}{y + \alpha^2 z}{x + \alpha y + \alpha^3 z}
$ &
130. $ \semsymmat{x}{y}{z}{\alpha^2 z}{x + z}{\alpha^2 z}{y + 2 z}{x + \alpha z}{y +
\alpha^3 z}{x + \alpha y + \alpha^3 z}
$ &
131. $ \semsymmat{x}{y}{z}{\alpha^5 z}{x + \alpha^2 z}{\alpha^7 z}{y + \alpha^6 z}{x +
\alpha^3 z}{y + \alpha^3 z}{x + \alpha y + \alpha^2 z}
$ \\ \\
132. $ \semsymmat{x}{y}{z}{\alpha^7 z}{x + \alpha^3 z}{\alpha z}{y + \alpha^2 z}{x +
\alpha^3 z}{y + \alpha^7 z}{x + \alpha y + \alpha^3 z}
$ &
133. $ \semsymmat{x}{y}{z}{z}{x + \alpha^6 z}{z}{y}{x + \alpha z}{y + \alpha z}{x +
\alpha y + \alpha^6 z}
$ &
134. $ \semsymmat{x}{y}{z}{0}{x + \alpha^7 z}{0}{y + 2 z}{x + \alpha z}{y}{x + \alpha y
+ 2 z}
$ \\ \\
%\end{array}
%\]
%
%\[
%\begin{array}{lll}
135. $ \semsymmat{x}{y}{z}{\alpha z}{x}{0}{y + \alpha^5 z}{x + \alpha^3 z}{y +
\alpha z}{x + \alpha y + \alpha^6 z}
$ &
136. $ \semsymmat{x}{y}{z}{0}{x + \alpha^7 z}{2 z}{y + \alpha^2 z}{x + \alpha^7 z}{y +
\alpha z}{x + \alpha y + \alpha^6 z}
$ &
137. $ \semsymmat{x}{y}{z}{\alpha^7 z}{x + \alpha^6 z}{\alpha^7 z}{y + z}{x +
\alpha z}{y + z}{x + \alpha y + 2 z}
$ \\ \\
138. $ \semsymmat{x}{y}{z}{0}{x + \alpha z}{\alpha^5 z}{y + z}{x + \alpha^5 z}{y +
\alpha^5 z}{x + \alpha^3 y}
$ &
139. $ \semsymmat{x}{y}{z}{0}{x}{0}{y + z}{x + \alpha z}{y}{x + \alpha^3 y + \alpha z}
$ &
140. $ \semsymmat{x}{y}{z}{\alpha^5 z}{x + \alpha z}{\alpha^7 z}{y + z}{x +
\alpha z}{y}{x + \alpha^3 y + \alpha^5 z}
$ \\ \\
141. $ \semsymmat{x}{y}{z}{\alpha^2 z}{x + \alpha^3 z}{\alpha^2 z}{y + \alpha^3 z}{x +
\alpha^3 z}{y + \alpha^5 z}{x + \alpha^3 y + \alpha z}
$ &
142. $ \semsymmat{x}{y}{z}{0}{x + \alpha^7 z}{\alpha^2 z}{y + \alpha^7 z}{x +
\alpha^7 z}{y + \alpha z}{x + \alpha^3 y + \alpha^5 z}
$ &
143. $ \semsymmat{x}{y}{z}{\alpha^6 z}{x + 2 z}{2 z}{y + \alpha^2 z}{x + \alpha^2 z}{y
+ z}{x + \alpha^3 y + \alpha^5 z}
$ \\ \\
144. $ \semsymmat{x}{y}{z}{\alpha z}{x + \alpha^2 z}{\alpha z}{y + \alpha z}{x +
\alpha^3 z}{y + 2 z}{x + \alpha^3 y + \alpha z}
$ &
145. $ \semsymmat{x}{y}{z}{\alpha^2 z}{x + \alpha^3 z}{z}{y + \alpha^2 z}{x + z}{y +
2 z}{x + \alpha^3 y + \alpha z}
$ &
146. $ \semsymmat{x}{y}{z}{\alpha^6 z}{x + \alpha^3 z}{\alpha^6 z}{y + \alpha^2 z}{x +
\alpha^3 z}{y + \alpha z}{x + \alpha^3 y + \alpha^2 z}
$ \\ \\
147. $ \semsymmat{x}{y}{z}{0}{x + \alpha^5 z}{2 z}{y + \alpha^6 z}{x + \alpha^5 z}{y +
\alpha^3 z}{x + \alpha^3 y + \alpha^2 z}
$ &
148. $ \semsymmat{x}{y}{z}{z}{x + \alpha^2 z}{z}{y}{x + \alpha^3 z}{y + \alpha^3 z}{x +
\alpha^3 y + \alpha^2 z}
$ &
\end{longtable}
\end{center}

\end{tiny}

%\clearpage

%\section{Conclusion}

%%%%%%%%%%%%%%%%%%%%%%%%%%%%%%%%%%%%%%%%%%%%
%%%%%%%%%%%%%%%%%%%%%%%%%%%%%%%%%%%%%%%%%%%%%

\end{document}